\documentclass [12pt]{article}
\usepackage{amssymb,amsmath, amsthm, amscd,ifthen}
\usepackage[T2A]{fontenc}
\usepackage[utf8]{inputenc}
\usepackage{mathtext}
\usepackage[english,russian]{babel}
\usepackage[dvips]{graphicx}
\usepackage{indentfirst}

\newtheorem{satz}{Theorem}
\newtheorem{proposition}[satz]{Proposition}
\newtheorem{theorem}[satz]{Theorem}
\newtheorem{lemma}[satz]{Lemma}

\newtheorem{conjecture}[satz]{Conjecture}

\newtheorem{remark}[satz]{Remark}

\def\Z{\mathbb {Z}}
\def\F{\mathbb {F}}

\def\a{\alpha}

\def\C{\mathbb{C}}

\def\d{\delta}
\def\o{\omega}
\def\({\big (}
\def\){\big )}

\def\G{\Gamma}

\def\ls{\leqslant}
\def\gs{\geqslant}
\def\dim{{\rm dim}}
\def\le{\leqslant}
\def\ge{\geqslant}
\def\_phi{\varphi}
\def\eps{\varepsilon}

\def\Gr{{\mathbf G}}
\def\FF{\widehat}

\def\ov{\overline}
\def\Spec{{\rm Spec\,}}

\def\la{\lambda}

\def\Bohr{{\rm Bohr}}

\begin{document}

\title{Number of $A+B\ne C$ solutions in abelian groups and application to counting independent sets in hypergraphs\footnote{This work was supported by the Russian Federation Government (Grant number 075-15-2019-1926).}}
\author{Aliaksei Semchankau\footnote{Lomonosov Moscow State University, Faculty of Mechanics and Mathematics, Department of Dynamical Systems Theory, Leninskie Gory, 1, Moscow, Russia; Steklov Mathematical Institute of Russian Academy of Sciences, Gubkina 8, Moscow, Russia. E-mail:aliaksei.Semchankau@gmail.com}, Dmitry Shabanov\footnote{Moscow Institute of Physics and Technology, Laboratory of Combinatorial and Geometric Structures, Institutskiy per. 9, Dolgoprudny, Moscow Region, Russia; HSE University, Faculty of Computer Science, Myasnitskaya Str. 20, Moscow, Russia. E-mail:dmitry.shabanov@phystech.edu}, Ilya Shkredov\footnote{Steklov Mathematical Institute of Russian Academy of Sciences, Gubkina 8, Moscow, Russia; IITP RAS,
Bolshoy Karetny per. 19, Moscow, Russia; Moscow Institute of Physics and Technology, Institutskiy per. 9, Dolgoprudny, Moscow Region, Russia. E-mail:ilya.shkredov@gmail.com}}

\date{}

\maketitle

\centerline{\bf\large Abstract}

\vspace{4mm}
The paper deals with a problem of Additive Combinatorics. Let $\Gr$ be a finite abelian group of order $N$. We prove that the number of subset triples $A,B,C\subset \Gr$ such that for any $x\in A$, $y\in B$ and $z\in C$ one has  $x+y\ne z$
 equals
$$
  3\cdot 4^N+N3^{N+1} + O((3-c_*)^N)
$$
for some absolute constant $c_*>0$. This provides a tight estimate for the number of independent sets in a special 3-uniform linear hypergraph and gives a support for the natural conjecture concerning the maximal possible number of independent sets in such hypergraphs on $n$ vertices.

\section{Introduction}
The paper deals with estimating the number of special subset triples in abelian groups. Let us
start with some motivation of the main result and
discuss the problem concerning counting the number of independent sets in hypergraphs. Recall that if $H=(V,E)$ is a  hypergraph, then its vertex subset $W\subset V$ is \emph{independent} if it does not contain complete edges from $E$, i.e. for any $A\in E$, $A\nsubseteq W$. Let $i(H)$ denote the total number of independent sets in $H$.

\subsection{Related work: counting independent sets in hypergraphs}
The classical problem of extremal graph theory asks what is the maximal number of in\-de\-pen\-dent sets in a $d$-regular graph $G$ on $n$ vertices. In 1991 Alon \cite[Section 5]{Alon} conjectured that if $n$ is divisible by $2d$ then the number of independent sets in this class of graphs is maximized when $G$ is a union of $n/(2d)$ disjoint copies of $K_{d,d}$, the complete bipartite graph with equal size part $d$. The conjecture was proved by Kahn \cite{Kahn} for bipartite graphs and, finally, by Zhao \cite{Zhao} for all $d$-regular graphs on $n$ vertices. So, if $G$ is a $d$-regular graph on $n$ vertices, then
$$
  i(G)\ls \left(i(K_{d,d})\right)^{\frac{n}{2d}}=\left(2^{d+1}-1\right)^{\frac{n}{2d}}.
$$
Note that the inequality holds even when $n$ is not divisible by $2d$.

\bigskip
It is quite natural to consider the same problem for hypergraphs, especially for the class of linear hypergraphs. Recall that a hypergraph $H=(V,E)$ is said to be \emph{linear} if every two of its distinct edges do not share more than one common vertex, i.e. for any $A,B\in E$, $A\ne B$, we have
$$
  |A\cap B|\ls 1.
$$
In
recent paper \cite{CPST} Cohen, Perkins, Sarantis and Tetali posed the following general question: \emph{which $d$-regular, $k$-uniform, linear hypergraph on a given number of vertices has the most number of independent sets}?

Suppose that $H=(V,E)$ is a $d$-regular $k$-uniform linear hypergraph on $n$ vertices. The general container method (see, e.g., \cite{SaxThom}) implies the following upper bound on the number of independent sets:
\begin{equation}\label{sax_thom}
  \frac{\log_2 i(H)}n\leqslant \frac {k-1}{k}+O_k\left(\frac{\log_2^{2(k-1)/k} d}{d^{1/k}}\right).
\end{equation}
It is known that the first order term in the exponent in \eqref{sax_thom} is correct, however
the second is expected to be improved.
Authors of \cite{CPST} formulated the following conjecture:
\begin{conjecture}\label{conj_cohen}For any $d$-regular $k$-uniform linear hypergraph $H$ on $n$ vertices the following holds 
  \begin{equation}\label{conjecture}
  \frac{\log_2 i(H)}n\leqslant \frac {k-1}{k}+\frac{\log_2 k}{kd}.
\end{equation}
\end{conjecture}
As in Alon's conjecture for graphs,
authors of \cite{CPST} provide some example of a small hypergraph, whose appropriate number of copies will give  bound \eqref{conjecture}. In the current paper we concentrate on the case of 3-uniform hypergraphs, so we will describe the construction from \cite{CPST} only for $k=3$.

Consider the following \emph{mod} hypergraph $H^{mod}_d=(V,E)$, where $V=V_1\sqcup V_2\sqcup V_3$, every $V_i$ is equal to the cyclic group $\mathbb{Z}/d\mathbb{Z}$ and
$$
  E=\left\{(x,y,z):\;x\in V_1,\;y\in V_2,\;z\in V_3,\;x+y=z \pmod d \right\}.
$$
Clearly, this hypergraph is 3-uniform, 3-partite, $d$-regular and linear (for each $x\in V_1$, $y\in V_2$ there is a unique $z\in V_3$
such that
$(x,y,z)\in E$). It was noticed in \cite{CPST} that
$$
  i(H^{mod}_d)\gs 3\cdot 4^{d}-3\cdot 2^d+1.
$$
This lower bound estimates only the independent sets that do not touch some of the parts $V_1$, $V_2$, $V_3$. But can we say that this bound is almost tight? One the main results of our paper answers this question positively.

\begin{theorem}\label{theorem_ind_sets}
We have
\begin{equation}\label{f:theorem_ind_sets_1}
  i(H^{mod}_d) =  3\cdot 4^d+O((4-c_*)^d)
\end{equation}
and moreover
\begin{equation}\label{f:theorem_ind_sets_2}
 i(H^{mod}_d) =  3\cdot 4^d+ d3^{d+1} + O((3-c_*)^d)
\end{equation}
for some absolute constant $c_*>0$.
\end{theorem}
As a consequence, we immediately obtain that if $H$ is a union of $n/(3d)$ disjoint copies of $H^{mod}_d$, then in view of \eqref{f:theorem_ind_sets_1}
$$
  \frac{\log_2 i(H)}n\leqslant \frac 1{3d}\log_2\left(3\cdot 4^d+O((4-4\varepsilon)^d)\right)=\frac 23+\frac{\log_2 3}{3d}+O((1-\varepsilon)^d)
$$
for some positive $\varepsilon$. This estimate strongly supports Conjecture 1, see inequality \eqref{conjecture}.
Surprisingly, that on the other hand, from our more refined asymptotic formula \eqref{f:theorem_ind_sets_2} it follows that
\[
	\frac{\log_2 i(H)}n \ge \frac 23+\frac{\log_2 3}{3d} + \Omega\left(\frac{3^d}{d4^d} \right) 
\]
and hence one {\it must} add exponentially small error terms to make Conjecture \ref{conj_cohen} correct.

\bigskip
We also would like to mention that the problem of estimating the number of independent sets has some natural extensions. One can count the strong independent sets (see \cite{CPST}, \cite{OrdRoth}, \cite{EJPR}) or general $j$-independent sets (see \cite{BalobShab}).

\subsection{Problem statement for abelian groups}
By analogy with $\mathbb{Z}/d\mathbb{Z}$ we can consider any finite abelian group $\Gr$ and ask the following question: \emph{what is the number of triples $A,B,C\subset \Gr$ such that there is no solution $x+y=z$, where $x\in A$, $y\in B$ and $z\in C$?} If $|\Gr|=N$, then clearly, this number is at least
$$
  3\cdot 4^{N}-3\cdot 2^N+1,
$$
because any triple with one empty subset always fits. If we take one of $A$, $B$ or $C$ equals a one--element set, that would add $N\cdot 3^{N+1}$ more triples, and therefore the lower bound is 
\begin{equation}\label{f:lower_3^N}
    3\cdot 4^{N} + N \cdot 3^{N+1} + O(2^N).
\end{equation}

Our first aim is to show that this lower bound is tight up to some summand exponentially smaller than $4^N$. Then we further elaborate our argument to find the correct exponent from \eqref{f:lower_3^N}. 

\bigskip
The organization of
this
paper is the following. In section 2 we consider the case when $\Gr$ is the prime field.  In the next section 3 we deduce the main result in the general case. The arguments in the prime case  are simpler but nevertheless, the scheme of the proof is similar to the case of general abelian group.   Finally, in sections 4 and 5 we will repeat the same scheme but obtain a stronger estimate for the error term, using a new result on structure of dense subsets $A,B$ of $\F_p$ with $A+B \neq \F_p$, see Proposition \ref{p:Semchankau_A+B} below.
A similar instrument was introduced for the first time in  \cite{Semchankau_A(A+A)}, where it has found already some applications to the structure of sets with small Wiener norm.
Also, it allows to estimate size of $A(A+A)$  for any $A\subseteq \F_p$, see \cite{Semchankau_A(A+A)}.
Thus this part of the paper has an independent interest for
Additive Combinatorics.

\section{The case of the prime field}


Let $\Gr$ be a finite abelian group, $\FF{\Gr}$ be  the dual group. It is well--known that $\FF{\Gr}$ is isomorphic to $\Gr$.
In this paper we use the same letter to denote a set $A\subseteq \Gr$ and  its characteristic function $A: \Gr \to \{0,1 \}$.
If $f,g:\Gr \to \C$ are functions, then we write
\[
    (f*g) (x) = \sum_{y\in \Gr} f(y) g(x-y) \,.
\]
The {\it sumset} of sets $A,B \subseteq \Gr$ is
\[
	A+B = \{ a+b ~:~ a\in A,\, b\in B \} \,.
\]
Similarly, one can define the {\it difference set} of $A$ and $B$.
Given  a prime $p$ we write $\F_p$ for the prime field.
All logarithms are to base $2.$ The signs $\ll$ and $\gg$ are the usual Vinogradov symbols.

\bigskip

Now suppose that $A,B\subseteq \Gr$ are two sets and put $\mathcal{C} = \mathcal{C} (A,B) =\Gr \setminus (A+B)$.
By the Cauchy--Davenport inequality, see, e.g., \cite[Theorem 5.4]{TV} if $\Gr = \F_p$ and $A+B\neq \F_p$, that is, $\mathcal{C}\neq \emptyset$, we have
\begin{equation}\label{ineq:C-D}
    |A+B| \ge |A| + |B| - 1 \,.
\end{equation}

\bigskip

We need one of the main results from \cite{Green_AP}.
Recall that for a given subset $\G \subseteq \FF{\Gr}$ and a number  $\eps \in (0,1)$ the set $\Bohr(\Gamma, \eps)$ is called a {\it Bohr neighbourhood} (or a {\it Bohr set}), see, e.g., \cite[Section 4.4]{TV} if
\[
	\Bohr(\Gamma, \eps) = \{ x\in \Gr  ~:~ | \chi(x) -1| \le \eps \,, \forall \chi \in \G \} \,.
\]
Size of $\G$ is called the {\it dimension} of $\Bohr(\Gamma, \eps)$, $\eps$ is the {\it radius} of $\Bohr(\Gamma, \eps)$ and it is
well--known the connection of size of Bohr sets with its dimension and radius, e.g., see \cite[Lemma 4.20]{TV}
\begin{equation}\label{f:Bohr_size}
	|\Bohr(\Gamma, \eps)| \ge (\eps/2\pi)^{|\G|} |\Gr|  \,.
\end{equation}
Now we are ready to formulate the required extracted  result from \cite{Green_AP} about shifts of Bohr sets in sumsets (actually one can check that the arguments work for composite $p$ and even for an arbitrary finite abelian group as well).
The quantity $\log^{1/4} p$ below is not so important and it has chosen just for convenience.

\begin{theorem}
    Let $p$ be a sufficiently large prime number and  $A,B$ be sets from  $\F_p$, $|A| \ge \alpha p$, $|B| \ge \beta p$,
    $\kappa = \sqrt{\alpha \beta} \ge \log^{-1/4} p$.
    Also, let $|\mathcal{C} (A,B)| = \gamma p$ and $\omega \le \gamma$ be a parameter, $\omega \ge \exp(-(\log^{1/4} p))$.
    Then there is a shift $x\in \F_p$ and a Bohr set $\mathcal{B} = \Bohr(\G,\frac{\kappa}{64 d})$, $|\G| = d\le 1000 \kappa^{-2} \log (1/\omega)$,
    such that
\[
	\left| (A + B)  \cap (\mathcal{B} + x)  \right| \ge \left( 1- \frac{32 \omega}{\kappa} \right) |\mathcal{B}| \,.
\]
\label{t:Green_AP}
\end{theorem}

A similar result on almost periodicity of convolutions was obtained in \cite{CLS}.
For the sake of the completeness we formulate a consequence of it.
The dependencies on the parameters in Theorems \ref{t:Green_AP}, \ref{t:CLS} are slightly different but in our regime ($\a,\beta \gg 1$)  this is  absolutely not important.

\begin{theorem}
	Let $\Gr$ be a finite abelian group, $N=|\Gr|$ and $A,B$ be sets from  $\Gr$, $|A| \ge \alpha N$, $|B| \ge  \beta N$,
	$\kappa = \sqrt{\alpha \beta}$.
	Also, let $q\ge 2$ and $\eps \in (0,1)$ be parameters.
	Then there is a shift $x\in \Z/N\Z$ and a Bohr set $\mathcal{B} = \Bohr(\G,c \eps)$, $c>0$ is an absolute constant, $|\G| \ll q/\eps^2$,
	such that
	\begin{equation}\label{f:CLS}
	\left| (A + B)  \cap (\mathcal{B} + x)  \right| \ge  \left( 1 - \left( \frac{\eps}{\kappa} \right)^q \right)|\mathcal{B}| \,.
	\end{equation}
\label{t:CLS}
\end{theorem}

Indeed, by the main result of  \cite[Theorem 1.2]{CLS} for almost periodicity of $f(x):= (A*B)(x)$, we have
\[
   \sum_{x\in \Gr} \sum_{t\in \mathcal{B}} \left( (A*B)(x+t) - (A*B)(x) \right)^{q} \le (\eps \kappa N)^q |\mathcal{B}| N
\]
and hence if \eqref{f:CLS} does not hold for any $x\in \Gr$, then we obtain a contradiction in view of the simple bound
\[
    \kappa^{2q} N^{q+1} \le \sum_{x\in \Gr} (A*B)^q (x) \,.
\]
\bigskip

Let us write $\#\{R\ |\ P(R)\}$ for the number of objects $R$, satisfying a property $P$. We simply write $\#\{R\}$, if the required property is clear from the context. Also, let us introduce the  following nonstandard notation. For sets $X, Y \subseteq \F_p$ we write $X \neq Y$ if the equation $x = y$ has no solutions in $x \in X, y \in Y$. 

Now we are ready to prove the main result of this section.

\begin{theorem}
    Let $p$ be a  sufficiently large prime number.
    Then the number of sets $A,B,C$ from $\F_p$ such that  $x+y\neq z$, $x\in A$, $y\in B$, $z\in C$
    equals
\begin{equation}\label{f:A+B=C}
    3 \cdot 4^{p} + O((4-c_*)^p)  \,,
\end{equation}
	where $c_*>0$ is an absolute constant.
\label{t:A+B=C}
\end{theorem}
\begin{proof}
    As we have noted above, if one takes $A,B$ or $C$ equals the empty set and the rest is an arbitrary, then we obtain
    $3 \cdot 4^p + O(2^p)$ of the required sets.
    The quantity $3 \cdot 4^p$ is the main term and our task is to estimate the rest.

    Suppose that all sets $A,B,C$ are non--empty.
    Put $a=|A|$, $b=|B|$, $c=|C|$.
	Clearly, $|\mathcal{C} (A,B)| \ge c$ hence by inequality \eqref{ineq:C-D}, we get $c \le p-|A+B| = p-a-b+1$ and, similar bounds hold for $a$ and $b$.
    We begin with a crude upper bound for
    the number of triples $A,B,C$ with small $a,b$ or $c$.
    Namely, for $M \le p/16$, say, one has
\begin{equation}\label{tmp:29.04_1}
    \sum_{a\le M} \binom{p}{a} \sum_{b=1}^{p} \binom{p}{b} \sum_{c=1}^{p-a-b+1}  \binom{p-a-b+1}{c}
    \le
    2^{p+1}
        \sum_{a\le M} \binom{p}{a} 2^{-a} \sum_{b=1}^p \binom{p}{b} 2^{-b}
            \le
\end{equation}
\begin{equation}\label{tmp:29.04_2}
            \le
            	4\cdot  3^p \binom{p}{M} 2^{-M}
            \le
            4\cdot  3^p \left( \frac{ep}{2M} \right)^M
            <
            4 (3.75)^p \,.
\end{equation}
	Thus $a,b,c \ge p/16$.

	Now we apply Theorem \ref{t:Green_AP} with the parameters $\alpha = \beta = \gamma = \kappa = 1/16$
	and let $\omega \le  2^{-10}$ be a sufficiently small
	number,
	which we will choose later (an alternative way is to use Theorem \ref{t:CLS}).
	Let $d$ and $\mathcal{B}$ be as in Theorem \ref{t:Green_AP}.	
	The number of all possible shifted Bohr sets is $p^{d+1}$. 
	We know that
	$C$ intersects $\mathcal{B}$  by at most  $2^9 \omega |\mathcal{B}| := E$ points.	Let us set $C' := C \cap \mathcal{B}$.
	The number of all sets of size at most $E$ does not exceed  $E\binom{p}{E}$.
	Put $q=|\mathcal{B}| - E \ge |\mathcal{B}|/2 >0$.
 Without loss of generality (since it will not affect the exponent) we assume that $0 \in B$, and therefore $A \neq C$. For fixed $a, b, c$ we obtain
	$$
	\#\{A, B, C\ |\ A + B \neq C\} \leqslant
	$$
	$$
	\leqslant
	\#\{\mathcal{B}\}\ 
	\#\{C'\}\ 
	\#\{C \setminus C' \subseteq \F_p\setminus \mathcal{B}\}\  
	\#\{A\ |\ A \neq C\}\ 
	\#\{B\ |\ B \neq C - A\}
	\leqslant 
	$$
	$$
	\leqslant
	p^{d+1} E\binom{p}{E}
	\binom{p-q}{c}\binom{p}{a}\binom{p-c-a+1}{b}.
	$$
	Summing it for all possible triples $a, b, c$ and considering obvious inequalities $c \leqslant p - q, a \leqslant p - c, b \leqslant p - c - a + 1$ we see that the number of all possible triples $A,B,C$
	is at most
\[
	\sigma
	:=
	p^{d+1} E \binom{p}{E} 
	\sum_{c=1}^{p-q} \binom{p-q}{c} \sum_{a=1}^{p-c} \binom{p}{a} \sum_{b=1}^{p-c-a+1} \binom{p-c-a+1}{b}
	\le
\]
\[
	\le
	2E p^{d+1} \binom{p}{E} \sum_{c=1}^{p-q} \binom{p-q}{c} 2^{p-c} \sum_{a=1}^{p-c} \binom{p}{a} 2^{-a}
	\le
	2E p^{d+1} \binom{p}{E} 3^p \sum_{c=1}^{p-q} \binom{p-q}{c} 3^{-c}
	\le
\]
\begin{equation}\label{tmp:29.04_3}
	\le
	p^{d+2} \left( \frac{ep}{E}\right)^E  \left( \frac{3}{4} \right)^q \cdot 4^p \,.
\end{equation}
	Put $l= \log (1/\o)$.
	By estimate \eqref{f:Bohr_size}, we know that $|\mathcal{B}| \ge q \gg p \exp(-Cl \cdot \log l)$, where $C>0$ is an absolute constant.
	Taking $\omega$ sufficiently small constant, we can attain
\begin{equation}\label{tmp:29.04_3'}
	\left( \frac{ep}{E}\right)^E \ll \left( \frac{C_1 \exp (C l\log l)}{\omega} \right)^{2^{10} \o q} \le \left( \frac{8}{7} \right)^q \,,
\end{equation}
	where $C_1>0$ is another absolute constant.
	Thus the number  $d \le 1000 \kappa^{-2} \log (1/\omega)$ is a constant and hence the multiple $p^{d+2}$ in \eqref{tmp:29.04_3} is negligible.
    Whence for sufficiently large $p$, we have, say,
$$
	\sigma \ll  \left( \frac{9}{10} \right)^q \cdot 4^p  \ll \left( \frac{9}{10} \right)^{p \exp(-Cl \cdot \log l)} \cdot 4^p = (4-c_*)^p \,.
$$
    This completes the proof.
$\hfill\Box$
\end{proof}

\section{The general case}

The case of an arbitrary finite abelian group $\Gr$ requires
more refined arguments and ge\-ne\-ra\-li\-za\-tions.
For example, inequality \eqref{ineq:C-D} is a particular case of Kneser's Theorem, see, e.g.,  \cite[Theorem 5.5]{TV}, which takes place in any abelian group.

\begin{theorem}
	Let $\Gr$ be an abelian group, and $A,B\subseteq \Gr$ be sets.
	Then
	\[
	|A+B| \ge |A+H| + |B+H| - |H| \,,
	\]
	where $H := \{ x\in \Gr ~:~ A+B+x = A+B \}$.
\label{t:Kneser}
\end{theorem}


Now we are ready to prove an analogue of Theorem \ref{t:A+B=C} and we will appeal to the proof of this result.

\begin{theorem}
	Let $\Gr$ be a finite abelian group, $N=|\Gr|$.
	Then the number of sets $A,B,C$ from $\Gr$ such that  $x+y\neq z$, $x\in A$, $y\in B$, $z\in C$
	equals
	\begin{equation}\label{f:A+B=C,N}
	3 \cdot 4^{N} + O((4-c_*)^N)  \,,
	\end{equation}
	where $c_*>0$ is an absolute constant.
	\label{t:A+B=C,N}
\end{theorem}
\begin{proof}
	As in the proof of Theorem \ref{t:A+B=C} (see calculations in \eqref{tmp:29.04_1},  \eqref{tmp:29.04_2}) we can assume that two sets from $A,B,C$ are large, say, $\Omega(N)$.
	Indeed, using the notation of the proof of this Theorem, we see the rest can be estimated as $2^N \binom{N}{M}^2 \le (3.75)^N$, say.
	Without losing of the generality, suppose that $A$ and $B$ are large, i.e., $|A|, |B| \ge c_0 N$ with an absolute constant $c_0>0$.
	Put  $H := \{ x\in \Gr ~:~ A+C+x = A+C \} \le \Gr$, and suppose firstly  that $h:=|H| \le c_1 N$, where $c_1>0$ is a sufficiently small absolute constant.
	Then Kneser's inequality gives us $|A+C| \ge |A|+|C|- |H|$.
	Hence combining this with  Theorem \ref{t:Green_AP} or Theorem \ref{t:CLS} and  acting  as in the proof of Theorem \ref{t:A+B=C}, we get
\[
	E p^{d+1} \binom{p}{E}  \sum_{c=1}^{N-q} \binom{N-q}{c} \sum_{a=1}^{N-c} \binom{N}{a} \sum_{b=1}^{N-c-a+h} \binom{N-c-a+h}{b}
		\le
	N^{d+2} 2^h \left( \frac{eN}{E}\right)^E  \left( \frac{3}{4} \right)^N \cdot 4^N
\]	
    (again $d$ is as in Theorem \ref{t:Green_AP}).
	Hence
	as in inequality \eqref{tmp:29.04_3'} we obtain the required asymptotic formula \eqref{f:A+B=C,N} taking sufficiently small $c_1$.
	Thus $h> c_1N$ and we can assume that $h<N$ because otherwise $A+C = \Gr$.
	Hence $h\le N/2$.
	Put $n=N/h \ge 2$, $n\in \Z$.
	Let $k_A$ be the number of different cosets $\Gr/H$, which intersects our random set $A$	(and, similarly, define $k_B$, $k_C$ for $B$ and $C$).
     From Kneser's Theorem we see that
     $k_A, k_B, k_C <n$
    otherwise
	the correspondent sumset
	coincides with the whole $\Gr$.
    Applying Theorem \ref{t:Kneser} again, we obtain  that $b \le (n- k_A - k_C + 1)h$.
	Using the arguments as in \eqref{tmp:29.04_1}, \eqref{tmp:29.04_2} one more time, combining with Theorem \ref{t:Kneser}, we derive a crude upper  bound for the number of possible $A,B$ and $C$
\[
	\sum_{k_A=1}^{n-1} \binom{n}{k_A} 2^{k_A h}
		\sum_{k_C=1}^{n-k_A} \binom{n}{k_C} 2^{k_C h}
			\sum_{k_B=1}^{n-k_A-k_C+1} \binom{n-k_A-k_C+1}{k_B} 2^{k_B h}
	\le
\]
\[
	\le
	n 2^{3n}
	\sum_{k_A=1}^{n-1}  2^{k_A h} \sum_{k_C=1}^{n-k_A} 2^{k_C h} \cdot  2^{(n-k_A-k_C+1)h}
	\le
	n^3 2^{3n} 2^{N+h} \le n^3 2^{3n} 2^{3N/2} \,.
\]	
	Since $h\ge c_1 N$,
	it follows that
	that $n\le c_1^{-1} = O(1)$ and hence the multiple $n^3 2^{3n}$ in the formula above is negligible.
	Also, it is easy to see that the number of subgroups in $\Gr$ of index at most $n = O(1)$ is $N^{O(1)}$
	(e.g., consider the canonical homomorphism of left cosets of $H$ in $\Gr$ to the symmetric group on $n$ letters) and this latter number is also negligible.
	This completes the proof.
$\hfill\Box$
\end{proof}

\section{An improvement}

In this section we obtain an improvement of Theorem \ref{t:A+B=C}, using other new tools, e.g., the Fourier transform on $\Gr$.
Our main argument works in the case of the prime field only although some statements hold  to be true for a general finite abelian group $\Gr$.

\bigskip

We denote the Fourier transform of a function  $f : \Gr \to \mathbb{C}$ by~$\FF{f},$ namely,
\begin{equation}\label{F:Fourier}
\FF{f}(\chi) =  \sum_{x \in \Gr} f(x) \overline{\chi(x)} \,,
\end{equation}
where $\chi \in \FF{\Gr}$ is an additive character on $\Gr$.
We rely on the following basic identities.
The first one is called the Plancherel formula and its particular case $f=g$ is called the Parseval identity
\begin{equation}\label{F_Par}
\sum_{x\in \Gr} f(x) \ov{g (x)}
=
\frac{1}{|\Gr|} \sum_{\chi \in \FF{\Gr}} \widehat{f} (\chi) \ov{\widehat{g} (\chi)} \,.
\end{equation}
Another  particular case of (\ref{F_Par}) is
\begin{equation}\label{svertka}
\sum_{x\in \Gr} |(f*g) (x)|^2
=
\sum_{x\in \Gr} \Big|\sum_{y\in \Gr} f(y) g(x-y) \Big|^2
= \frac{1}{|\Gr|} \sum_{\chi \in \FF{\Gr}} \big|\widehat{f} (\chi)\big|^2 \big|\widehat{g} (\chi)\big|^2 \,.
\end{equation}
and the identity
\begin{equation}\label{f:inverse}
f(x) = \frac{1}{|\Gr|} \sum_{\chi \in \FF{\Gr}} \FF{f}(\chi) \chi(x)
\,.
\end{equation}
is called the inversion formula.

\bigskip

Our proof is based on several auxiliary statements. 
First of all, we need a consequence of the following result of J.M. Pollard \cite{Pollard} (also, see \cite[Corollary 1.2]{Tao_Pollard}).
Given sets $A,B \subseteq \F_p$  and $\eps \in (0,1)$ we write $A+_\eps B$ for the set of $x\in \F_p$, having at least $\eps p$ representations as sum of $a+b$, $a\in A$, $b\in B$.

\begin{theorem}
	Let $A,B \subseteq \F_p$ be sets and $\eps \in (0,1)$ be a real number, such that $\sqrt{\eps} p < |A|, |B|$.
	Then
\begin{equation}\label{f:Pollard}
	|A+_\eps B| \ge \min \{p, |A|+|B|\} -  2p \sqrt{\eps} \,.
\end{equation}
\label{t:Pollard}
\end{theorem}

We have the following combinatorial observation: if sets $X, Y, Z \subseteq \F_p$ satisfy $X + Y \neq Z$, then they satisfy $X \neq Z - Y$ as well. Let us prove the  following robust version of this
truism. 
\begin{lemma}\label{l:struct}
    Let  $X, Y, Z$ be nonempty subsets of  $\F_p$ and $\delta \in (0, 1)$ be a parameter such that $Z \neq X +_{\delta} Y$. Also, let us assume that $|X| \geqslant \eps p$  for a certain  $\eps \in (0, 1)$. Let $T > 1$ be a parameter. Then there exists $X' \subseteq X, |X'| \leqslant |X|/T$ such  that
    $$
    X\setminus{X'} \neq Z -_{\delta T/\eps}Y \,.
    $$
\end{lemma}
\begin{proof}
    Let $\eta := \delta T/\eps$.
    Suppose that for a certain  $X' \subseteq X$ one has  $X' \subseteq Z -_{\eta}Y$. Then we have at least $|X'|\eta p$ triples $(x, y, z)$ such that $x + y = z$. By the pigeonhole principle we see that there is  $z_0$ with at least $|X'|\eta p/|Z|$ representations $z_0 = x + y$. By assumption  $Z \neq X +_{\delta}Y$ and hence  we derive the inequality $|X'|\eta p /|Z| \leqslant \delta p$, which gives $|X'| \leqslant |Z|\delta /\eta \leqslant p\delta/\eta \leqslant |X|/T$.
$\hfill\Box$
\end{proof}

\bigskip 

Secondly,
we need the well--known Chang's Theorem see, e.g., \cite{TV}.
Recall that for a set $A\subseteq \Gr$ and $\eps \in (0,1]$ the {\it spectrum} $\Spec_\eps (A)$ is defined as
$$
\Spec_\eps (A) := \{ \chi\in \FF{\Gr} ~:~ |\FF{A} (\chi)| \ge \eps |A| \} \,.
$$
Also, the {\it additive dimension} $\dim (A)$ of a set $A \subseteq \Gr$ is size of the maximal dissociated subset of $A$, i.e., size of maximal $\Lambda \subseteq A$ such that any equation $\sum_{\la\in \Lambda} \eps_\la \la = 0$ with $\eps_\la \in \{0,1,-1\}$ implies $\eps_\la = 0$, $\forall \la \in \Lambda$.

\begin{theorem}
	Let $\Gr$ be a  finite abelian group, $A \subseteq \Gr$, and $\eps \in (0,1]$ be a real number.
	Then
\[
	\dim (\Spec_\eps (A)) \le 2 \eps^{-2} \log (|\Gr|/|A|) \,.
\]
\label{t:Chang}
\end{theorem}

Now let us formulate a rather general and simple result on level sets of an arbitrary function.

\begin{lemma}
	Let $\Gr$ be an abelian group, $\d$, $\a$, $\beta$  be real numbers, $\a \le \beta$, $Y$ be a set, and $f,g : \Gr \to \C$ be functions.
	Put
	\[
		X = \{ x\in \Gr ~:~ f(x) \in [\a, \beta] \} \,,
	\]
	and
	\[
		E^{-} = \{ x\in \Gr ~:~ f(x) \in [\a-\delta, \a) \}\,,
		\quad
		E^{+} = \{ x\in \Gr ~:~ f(x) \in (\beta, \beta+\delta] \} \,.
	\]
	Suppose that for all $x\in \Gr\setminus Y$ one has $\| f-g \|_\infty \le \delta$
	and
	write $g(x) = \sum_{j} g(x) S_j (x)$, where $g$
	differs
	by
	at most $\d$ on some disjoint sets $S_j$.
	Then
\begin{equation}\label{f:Semchankau}
	\bigsqcup_{j ~:~ S_j \cap (X\setminus Y) \neq \emptyset} S_j =  ((X\setminus Y) \cup Y') \cup E' \cup E^{''} \,,
\end{equation}
	where  $E' \subseteq  E^{-}$, $E'' \subseteq  E^{+}$, $Y' \subseteq Y$ are some sets.
\label{l:Semchankau}
\end{lemma}
\begin{proof}
	Put $X' = X\setminus Y$, $X'' = \bigsqcup_{j ~:~ S_j \cap X' \neq \emptyset} S_j$.
	Clearly, $X' \subseteq X''$ because we can assume that $\bigsqcup_{j} S_j = \Gr$.
	Our task is to prove that $X''$ coincides with the right--hand side of \eqref{f:Semchankau}.
	If
	$x\in X''$,
	then $f(x) \in [\a-\d, \beta+\d]$ and hence
	$X'' \subseteq X \sqcup E^{'} \sqcup E^{''}$, where $E' := E^{-} \cap X''$ and $E'' := E^{+} \cap X''$.
	Let $Y' := Y\cap  X''$.
	Then $X'' \subseteq ((X\setminus Y) \cup Y') \cup E^{'} \cup E^{''}$ but the
	reverse
	inclusion is obvious (recall that $X'=X\setminus Y \subseteq X''$).
	This completes the proof.
$\hfill\Box$
\end{proof}

\bigskip  	

\begin{remark}
	In follows from the proof of Lemma \ref{l:Semchankau} that the sets $S_j$ do not need to be disjoint and one can consider any covering of $\Gr$ by a collection of 
	sets $S_j$.
	Nevertheless,  then we lose
	disjointedness in \eqref{f:Semchankau} and it is important sometimes to keep it, see \cite{Semchankau_A(A+A)}.
\end{remark}

Theorem \ref{t:Pollard}, combining with Lemma \ref{l:Semchankau} give us a useful result.

\begin{proposition}\label{p:Semchankau_A+B}
	Let $A,B \subseteq \F_p$ be sets, $|A| = \a p$, $|B| = \beta p$, $\d, \eps \in (0,1]$ be real parameters, and $C = \mathcal{C} (A,B) \neq \emptyset$.
	Then there are sets $Y$, $W$, $C\setminus Y \subseteq W$ such that the following holds
	$$
		|Y| \le \eps^2 \alpha^2 \beta \d^{-2} p \,,
	\quad \quad
		W = \bigcup_{j=1}^t S_j \,,
	$$
	where $S_j$ are shifts of a  Bohr set $\mathcal{B}$ of dimension $2 \eps^{-2} \log \a^{-1}$ and radius $\d (\a \beta)^{-1/2}$,
\begin{equation}\label{e:t}
		t\le \exp(O(\eps^{-2} \log \a^{-1} \cdot \log (\eps \d)^{-1})) \,,
\end{equation}
$$
	W \neq A +_{\delta}B,
$$
and
$$
	|W| \le p - |A| -|B| + 2p\sqrt{\d} \,.
$$
\end{proposition}
\begin{proof}
	Let $f(x) = (A*B)(x)$ and $g(x) = p^{-1} \sum_{r\in \Spec_\eps(A)} \FF{f} (r) e(rx)$.
	Due to formulae \eqref{F_Par}, \eqref{svertka}, we have
\[
	\sum_x (f(x) - g(x))^2 \le p^{-1} \sum_{r\notin \Spec_\eps(A)} |\FF{A} (r)|^2 |\FF{B} (r)|^2 \le \eps^2 |A|^2 |B|
\]
	and hence
$
	\| f(x) - g(x) \|_\infty \le \delta p
$
	outside of a set $Y$ of size $|Y| \le \eps^2 \alpha^2 \beta \d^{-2} p$.
	Apply  Lemma \ref{l:Semchankau} with the parameters $\a=\beta =0$ to the constructed functions $f$, $g$ and to the set $Y$.
	Then $E'$ is empty and $|E''| \le p - |A| -|B| + 2 \sqrt{\delta} p$ thanks to  Theorem \ref{t:Pollard}.
	Thus we put $W:=E''= \{ x ~:~ f(x) \le \delta \}$. 
	In other words, $W \neq A +_{\delta}B$. 
	Further by Chang's Theorem \ref{t:Chang} there exists a set $\Lambda$, $|\Lambda| \le 2 \eps^{-2} \log \a^{-1}$ such that any element of $\Spec_\eps(A)$ can be expressed as $\sum_{\la \in \Lambda} \xi_\la \la$,  $\xi_\la \in \{0,1,-1\}$.
	Put $\mathcal{B} = \Bohr (\Lambda, \zeta/ |\Lambda|)$ with $\zeta = \delta (\a \beta)^{-1/2}$.
	Then for any $b\in \mathcal{B}$ and for an arbitrary $x\in \F_p$, one has
	$$
		|g(x+b) - g(x)| \le p^{-1} \sum_{r\in \Spec_\eps(A)} |\FF{A} (r) \FF{\mathcal{B}} (r)| |e(rb)-1|
		\le
		\zeta p^{-1} \sum_{r} |\FF{A} (r)| |\FF{\mathcal{B}} (r)|
		\le \sqrt{\a \beta} \zeta p  = \delta p \,,
	$$
	where we have used the triangle inequality twice, the Cauchy--Schwarz inequality and finally,  formula \eqref{F_Par}.
	It means that the function $g$ differs at most $\d p$ on any shift of $\mathcal{B}$.
	
	Now we want to find some $s_1, s_2, \ldots, s_t$ so that shifts $s_i +\mathcal{B}$ cover the whole group $\F_p$. Let us consider a Bohr set $\mathcal{B}'$, identical to $\mathcal{B}$, but with twice smaller radius. By the definition of Bohr sets it is clear that $\mathcal{B}' - \mathcal{B}' \subseteq \mathcal{B}$.
	Let $S := \{ s_1, s_2, \ldots, s_t\}$ be a maximal set such that shifts $s_i + \mathcal{B}'$ do not overlap.
	Due to its maximality, for any $s\in \F_p$ there exists such $s_i$ so that $s + \mathcal{B}'$ intersects with $s_i + \mathcal{B}'$, and therefore $s \in s_i + \mathcal{B}' - \mathcal{B}' \subseteq s_i + \mathcal{B}$. Therefore, shifts $s_i + \mathcal{B}$ cover the whole group $\F_p$. 
	Since $s_i + \mathcal{B}'$ do not overlap, we have $p \geqslant |S + \mathcal{B}'| = |S||\mathcal{B}'|$, and therefore, by estimate \eqref{f:Bohr_size}, we see that the number $t$ of such shifts is at most $p/|\mathcal{B}'| \leqslant (4 \pi |\Lambda| \zeta^{-1})^{|\Lambda|}$.
	This completes the proof.
$\hfill\Box$
\end{proof}		

\bigskip 

Repeating the proof above, one can obtain the following slightly different 
\begin{proposition}\label{p:Semchankau_A+B_alt}
    	Let $A,B \subseteq \F_p$ be sets, $|A| = \a p$, $|B| = \beta p$, $\eta, \d, \eps \in (0,1]$ be real parameters, and $C$ be a set, $C \neq A +_{\eta} B$.
    
	Then there are sets $Y$, $W$, $C\setminus Y \subseteq W$ such that the following holds
	$$
		|Y| \le \eps^2 \alpha^2 \beta \d^{-2} p \,,
	\quad \quad
		W = \bigcup_{j=1}^t S_j \,,
	$$
	where $S_j$ are shifts of a  Bohr set $\mathcal{B}$ of dimension $2 \eps^{-2} \log \a^{-1}$ and radius $\d (\a \beta)^{-1/2}$,
\begin{equation}\label{e:t_alt}
		t\le \exp(O(\eps^{-2} \log \a^{-1} \cdot \log (\eps \d)^{-1})) \,,
\end{equation}
$$
	W \neq A +_{(\eta + \d)}B,
$$
and
$$
	|W| \le p - |A| -|B| + 2p\sqrt{\eta + \d} \,.
$$
\end{proposition}

\bigskip 

We write Chernoff Bound, according to \cite{TV}:
 \begin{theorem} [\bf Chernoff's inequality]\label{t:chernoff}
 Assume that $X_1, \ldots, X_n$ are jointly independent random variables where 
 $|X_i - \mathbb{E}X_i| \leqslant 1$
 for all $i$.
 Set $X := X_1 + \ldots + X_n$ and
 let $\sigma := \sqrt{\mathrm{Var}(X)}$ be the standard deviation of $X$. Then for any $\lambda > 0$

$$
\mathbb{P}(|X - \mathbb{E}X|
\geqslant 
\lambda\sigma) 
\leqslant 
2\max(
e^{-\lambda^2/4}, 
e^{-\lambda\sigma/2})
$$
 
 \end{theorem}
 
 \bigskip

Now we are ready to improve Theorem \ref{t:A+B=C}.

\begin{theorem}
	Let $p$ be a  sufficiently large prime number and $\eps \in (0,1]$ be any real parameter.
	Then the number of sets $A,B,C$ from $\F_p$ such that  $x+y\neq z$, $x\in A$, $y\in B$, $z\in C$
	equals
	\begin{equation}\label{f:A+B=C_new}
	3 \cdot 4^{p} + 3p \cdot 3^{p} + O((3 - c_{*})^p) \,,
	\end{equation}
	for some absolute constant $c_* > 0$.
	\label{t:A+B=C_new}
\end{theorem}
\begin{proof}
    When one of the sets $A, B, C$ is empty we obtain the term $3 \cdot 4^{p}$, and when one of them contains just one element, we get the term $3p \cdot 3^{p}$. From now let us assume that all sets contain at least two elements. 
    We will consider two cases: when one of the sets is small, and when all $A, B, C$ satisfy $|A|, |B|, |C| \gg \eps p$, where $\eps$ is a sufficiently small constant.
    
    Let us assume that $2 \leqslant |B| \ll \eps p$. Without loss of generality, we can assume that $B$ contains elements $0$ and $-1$ (by multiplying and shifting $A, B, C$ by the same constants). Therefore, $A \neq C \cup (C+1)$. The number of pairs $A, C$ is at most
    \begin{equation}\label{eq:ac}
        \sum_{c = 1}^{p}\sum_{C:\ |C| = c}
    \#\{A : A \neq C \cup (C+1)\} = \ 
    \sum_{c = 1}^{p}\sum_{C:\ |C| = c}
    2^{p - |C \cup (C+1)|}  \,.
    \end{equation}
  
    If we write $2^{p - |C|}$ instead of $2^{p - |C\cup (C+1)|}$, then the above sums up to $3^p$. However, $|C \cup (C+1)| = |C| + |C+1| - |C \cap (C+1)| = 2|C| - |C \cap (C+1)|$, and we expect $|C \cap (C+1)|$ to be equal roughly $\gamma^2 p$ in average, for $|C| = \gamma p$. Using Chernoff bound \eqref{t:chernoff} we can estimate the number of cases when $|C \cap (C+1)| \approx \gamma^2 p$
    (the latter has exponentially small probability), and we obtain estimate $(3 - c_*)^p$  for sum \eqref{eq:ac}  in this case. 

    Summing it for all $B, 2 \leqslant \min(|A|, |B|, |C|) \ll \eps p$ we obtain
    $$
    \#\{A, B, C : A + B \neq C\} 
    \leqslant 
    \#\{B\}(3 - c_*)^p
    \leqslant
    \exp(O(\eps p))(3 - c_*)^p 
    \leqslant
    (3 - c_* + O(\eps))^p.
    $$

    \bigskip

    Now let us consider the case, when all $A, B, C$ satisfy $|A|, |B|, |C| \gg \eps p$.

	Let us apply Proposition \ref{p:Semchankau_A+B} to sets $A, B, C$ with the parameters $\d := \d_C$ and $\eps := \eps_C$ ($\eps_C$ and $\d_C$ will  be defined later). 
	We find sets $W_C$, $Y_C$, so that 
	$$C \setminus Y_C \subseteq W_C ,\ \ W_C \neq A +_{\d_C}B.
	$$
	
	Let us choose some $T > 1$. By Lemma \ref{l:struct} we have $A\setminus A' \neq W_C -_{\delta_C T/\eps}B$ for some $A'$ of size at most $|A|/T$. Let $\d_A := \delta_C T/\eps$. Put $A_1 := A \setminus A'$. Apply Proposition \ref{p:Semchankau_A+B_alt} to the sets $A_1, W_C, B$ with the parameters $\eps := \eps_A, \eta := \d := \d_A$ (again $\eps_A$ will be defined later). We see that for some $W_A, Y_A$ one has 
	$$A_1 \setminus Y_A \subseteq W_A,\ \ W_A \neq W_C -_{2\d_A}B.
	$$
	
	In a similar way by Lemma \ref{l:struct} we have $B \setminus B' \neq W_C -_{2\d_A T /\eps} W_A$ for some $B'$ of size at most $|B|/T$. Let us put $\d_B := 2\d_A T /\eps$. Let $B_1 := B \setminus B'$. By Proposition \ref{p:Semchankau_A+B_alt} applied to the sets $B_1, W_A$, and $W_C$ with parameters $\eps := \eps_B$ and $\eta := \d := \d_B$ ($\eps_B$ will be defined later) we obtain some $W_B$ and $Y_B$, so that 
	$$
	B_1 \setminus Y_B \subseteq W_B,\ \ W_B \neq W_C -_{2\d_B} W_A.
	$$
	
	We know that all $Y_X$ (with $X = A, B$ or $C$) satisfy the rough inequality $|Y_X| \leqslant (\eps_X/\d_X)^2 p$.  Let us put  each $\eps_X$ to be equal $\eps \d_X$. Now we have bounds $|Y_X| \leqslant \eps^2 p$ for $X = A, B, C$. It is easy to see that
	$$
	\#\{Y_X\} \ll \exp(O(p\eps)).
	$$
	
	Each $W_X$ is defined by a collection of at most $t:=\exp(O(\eps_X^{-2} \log \eps^{-1}\log(\eps\d_X)^{-1}))$ shifts of  Bohr sets $S_j :=  \mathcal{B}_X + s_j$, $s_j \in \F_p$. We now guarantee that all $\d_X$ (and therefore $\eps_X$) do satisfy $\d_X \gg \eps^{\Omega(1)}$, by setting $T := 1/\eps^2$ and 
	$\d_C = \eps^{10}$ (notice that $\d_A = \d_CT/\eps = \eps^7$, and $\d_B = 2\d_A T/\eps = 2\eps^4$). From this we obtain $t \ll \exp(O(\eps^{-\Omega(1)}))$.
	We know that the dimension of our  Bohr set $\mathcal{B}_X$ is at most $\dim(\mathcal{B}_X), \dim(\mathcal{B}_X) \ll \eps_X^{-2} \log \eps^{-1} \ll \eps^{-\Omega(1)}$.
	Hence the number of possible sets $S_j = \mathcal{B}_X + s_j$ is at most 
	$$
	\#\{\mathcal{B}_X\} \#\{s_1, \dots, s_t\} 
	\ll 
	\binom{p}{\dim(\mathcal{B}_X)}
	\binom{p}{t}
	\ll 
	p^{\dim(\mathcal{B}_X) + t} 
	\ll 
	\exp(O(\eps p)).
	$$
	From here we obtain 
	$$
	\#\{W_X\} 
	\ll 
	\#\{\mathcal{B}_X\} \#\{s_1, \ldots, s_t\} 
	\ll 
	\exp(O(\eps p))
	$$
	for all $X = A, B, C$.
 	
 	Clearly, since $|A_1| \leqslant |A|/T \ll \eps^2p, |B_1| \ll \eps^2 p$, it follows that
 	$$\#\{A_1\}, \#\{B_1\} \ll 
 	\binom{p}{\eps^2 p} \ll \exp(O(\eps p)) \,.$$
 	
	Finally, from the formula $W_B \neq W_C -_{2\d_B} W_A$ and estimate \eqref{f:Pollard} from Theorem \ref{t:Pollard} we obtain that 
	$$
	|W_A| + |W_B| + |W_C| \leqslant p + 2\sqrt{2\d_B}p = p + O(\eps^2p)
	$$
	For given sets $W_A, W_B, W_C$, satisfying $|W_A|+|W_B|+|W_C| \leqslant p + O(\eps^2 p)$, and given $Y_A, Y_B, Y_C$ the number of triples  $A, B, C$ is bounded by
    $$
    \#\{A'\} \#\{A_1 : A_1 \setminus Y_A \subseteq W_A \}\ 
    \#\{B'\} \#\{B_1 : B_1 \setminus Y_B \subseteq W_B\}\ 
    \#\{C : C \setminus Y_C \subseteq W_C\},
    $$
    which is at most 
    $$
    \exp(O(\eps p))
    2^{|W_A|}2^{|W_B|}2^{|W_C|} \ll
    2^{p + O(\eps p)}.
    $$
    The number of sets $W_A, W_B, W_C, Y_A, Y_B, Y_C$ 
    can be estimated as 
$$
\#\{W_A\}\#\{Y_A\}\ 
\#\{W_B\}\#\{Y_B\}\ 
\#\{W_C\}\#\{Y_C\}
\ll 
\exp(O(\eps p)),
$$
and the final result is $2^{p + O(\eps p)}$. This completes the proof.
$\hfill\Box$
\end{proof}	

   \begin{remark}
    It is highly likely, that when $2 \leqslant |B| \ll \eps p$, one can estimate $\#\{A + B \neq C\}$ as 
    $$
    \exp(O(\eps p))
    \max_{\gamma \in (0, 1)}
    \binom{p}{\gamma p}
    2^{p - 2\gamma p + \gamma^2 p} 
    = 
    \exp(O(\eps p))
    \max_{\gamma \in (0, 1)} 
    e^{(-\gamma\ln{\gamma} - (1 - \gamma)\ln{(1 - \gamma)} + (1 - \gamma)^2\ln{2})p},
    $$
    which is optimized for $\gamma = 0.2653\ldots$ and gives $O(2.5926\ldots + O(\eps))^p$.
    \end{remark}

    \begin{remark}
It is also highly likely, that 
consider separately the cases $\min{(|A|, |B|, |C|)}$ $= 2, 3, 4, \ldots, k$, we can obtain bound on $\#\{A, B, C : A + B \neq C\}$ of the form 
$$
3\cdot 4^p + 
3p\cdot 3^p + 
Q_2(p)\lambda_2^p + 
Q_3(p)\lambda_3^p + 
\ldots + 
Q_{k}(p)\lambda_{k}^p +
O((\lambda_{k} - c_*)^p),
$$
where $4 > 3 > \lambda_2 > \lambda_3 > \ldots > \lambda_{k} > c_* > 0$ are absolute constants, and $Q_i$ are some polynomials on $p$.
\end{remark}

\section{An improvement in the general case}

The case of an arbitrary finite abelian group $\Gr$ requires
more refined arguments and ge\-ne\-ra\-li\-za\-tions.
For example, we have  more general version of inequality \eqref{f:Pollard} (see \cite[Corollary 6.2]{GreenRusza})
which follows by the same argument 
 from the generalization of Pollard Theorem (see, e.g., \cite[Proposition 6.1]{GreenRusza} or \cite[Theorem 1]{HamidouneSerra}).

\begin{theorem}
	Let $\Gr$ be an abelian group, and $A,B\subseteq \Gr$ be sets. 
	Let $\eps \in (0, 1)$ be such that $\sqrt{\eps}|\Gr| < |A|, |B|$.
	Let $H$ be a maximal proper subgroup of $\Gr$. 
	Then
	\[
	|A+_{\eps}B| \ge \min(|\Gr|, |A| + |B| - |H|) - 3\sqrt{\eps}|\Gr| \,,
	\]
\label{t:general_Pollard}
\end{theorem}

\begin{remark}\label{r:wa_wb_wc}
We will use this Theorem just to conclude the following observation. Let $W_a, W_b, W_c$ be some subsets of $\Gr$ of size at least $\eps p$. Let us also assume they satisfy $W_a \neq W_b +_{\d} W_c$ for some $\delta \ll \eps^2$. Then we have $|W_a| \leqslant |\Gr| - |W_b +_{\d}W_c|$.
Thanks to Theorem \ref{t:general_Pollard}
the last quantity  is at most $|\Gr| + |H| + 3\sqrt{\d}|\Gr| - |W_b| - |W_c|$ 
and this implies 
$$
|W_a| + |W_b| + |W_c| \leqslant |\Gr| + |H| + 3\sqrt{\d}|\Gr| \leqslant (3/2 + O(\eps))|\Gr|.
$$
\end{remark}


Now we are ready to prove an analogue of Theorem \ref{t:A+B=C} and we will appeal to the proof of this result.

\begin{theorem}
	Let $\Gr$ be a finite abelian group, $N=|\Gr|$.
	Then the number of sets $A,B,C$ from $\Gr$ such that  $x+y\neq z$, $x\in A$, $y\in B$, $z\in C$
	equals
	\begin{equation}
	3 \cdot 4^{N} + 3N\cdot 3^{N} + O((3 - c_*)^N)  \,,
	\end{equation}
	where $c_*>0$ is an absolute constant.
\end{theorem}
\begin{proof}
	As in the proof of Theorem \ref{t:A+B=C_new} we split the problem into two cases, namely when one of the sets $A,B,C$ is small, and when all the sets have sizes at least $\eps N$ for a certain  $\eps \in (0, 1)$.
	As in the proof of Theorem \ref{t:A+B=C_new} we see that 
	the first case gives us 
	at most $3 \cdot 4^{N} + 3N\cdot 3^{N} + O((3 - c_*)^N)$ of such triples.
	
	
	In the second case we repeat the same argumentation as for $\Gr = \F_p$ but we need to change it slightly, since we cannot use the inequality from Theorem \ref{t:Pollard} in the general case. Replacing usage of Theorem \ref{t:Pollard} with the more general Theorem \ref{t:general_Pollard} and taking into account Remark \ref{r:wa_wb_wc} we obtain for the sets $W_a, W_b, W_c$ that $|W_a| + |W_b| + |W_c| \leqslant (1.5 + O(\eps))N$, which changes the bound in the second case from $2^{N + O(\eps)N}$ to $2^{1.5N + O(\eps)N} = (2\sqrt{2} + O(\eps))^{N}$, and the proof is complete.

$\hfill\Box$
\end{proof}

\begin{remark}
Let us take a finite abelian group $\Gr$, $N = |\Gr|$ which has a subgroup $H$ of index $2$. Then for any $A, B \subseteq H, C \subseteq \Gr \setminus H$ we have $A + B \neq C$ and $\#\{A, B, C : A + B \neq C\} \gg 2^{|H|}2^{|H|}2^{|H|} = 2^{3N/2} = (2\sqrt{2})^N$. Therefore, one cannot obtain a smaller exponent 
at least in 
the second case of the proof.
\end{remark}

\renewcommand{\refname}{References}

\end{document}